\documentclass{article}
\usepackage{amsmath}
\usepackage{amsfonts}
\usepackage{amssymb}
\usepackage{graphicx}

\newtheorem{theorem}{Theorem}[section]

\newtheorem{lemma}[theorem]{Lemma}

\newtheorem{remark}[theorem]{Remark}

\newenvironment{proof}[1][Proof]{\noindent \textbf{#1.} }{\  \rule{0.5em}{0.5em}}
\numberwithin{equation}{section}

\newcommand{\abs}[1]{\left\vert#1\right\vert}
\newcommand{\set}[1]{\left\{#1\right\}}
\newcommand{\Real}{\mathbb R}

\newcommand{\Oo}{\Omega}
\newcommand{\oo}{\omega}

\newcommand{\F}{\mathcal{F}}

\newcommand{\Hh}{\mathcal{H}}
\newcommand{\Ss}{\mathcal{S}}

\newcommand{\1}{\mathbf 1}
\newcommand{\Ee}{\mathbf{E}}
\newcommand{\ul}{\underline}
\begin{document}

\title{A uniqueness theorem for solution of BSDEs}
\author{Guangyan JIA\thanks{%
The author thanks the partial support from the National Basic
Research Program of China (973 Program) grant No. 2007CB814901
(Financial Risk) and the National Natural Science Foundation of
China, grant No. 10671111.
email address: jiagy@sdu.edu.cn}\\
School of Mathematics and System Sciences\\
Shandong University, Jinan, Shandong, 250100, P.R.China}
\maketitle

{\small \textbf{Abstract.} In this note, we prove that if $g$ is
uniformly continuous in $z$, uniformly with respect to $(\oo,t)$ and
independent of $y$, the solution to the backward stochastic
differential equation (BSDE) with generator $g$ is unique.}

\section{Introduction}

One dimensional BSDEs are equations of the following type defined on
$[0,T]$:
\begin{equation}\label{equ0}
y_t=\xi+\int_t^T g(s,y_s,z_s)\,ds-\int_t^T z_s\,dW_s,\qquad 0\le
t\le T,
\end{equation}
where $W$ is a standard $d$-dimensional Brownian motion on a
probability space $(\Oo,\F,(\F_t)_{0\le t\le T},P)$ with
$(\F_t)_{0\le t\le T}$ the filtration generated by $W$. The function
$g:\Oo\times[0,T]\times\Real\times\Real^d\to\Real$ is called
generator of (\ref{equ0}). Here $T$ is the terminal time, and $\xi$
is a $\Real$-valued $\F_T$-adapted random variable; $(g,T,\xi)$ are
the parameters of (\ref{equ0}). The solution $(y_{t},z_{t}%
)_{t\in \lbrack0,T]}$ is a pair of $\F_t$-adapted and square
integrable processes.


Nonlinear BSDEs were first introduced by Pardoux and
Peng~\cite{PP1}, who proved the existence and uniqueness of a
solution under suitable assumptions on $g$ and $\xi$, the most
standard of which are the Lipschitz continuity of $g$ with respect
to $(y,z)$ and the square integrability of $\xi$. An interesting and
important question is to find weaker conditions rather than the
Lipschitz one, under which the BSDE (\ref{equ0}) still has a unique
solution. As a matter of fact, there have been several works, such
as Pardoux and Peng \cite{pardouxP1992}, Kobylanski \cite{KB} and
Briand-Hu \cite{briandH2007}, etc. In this note, we will give a new
sufficient condition for the uniqueness of the solution to BSDEs.

In fact, this problem came from a lecture given by Peng at a seminar
of Shandong University on Oct. 2005. In his lecture, Peng
conjectured that if $g$ is H\"{o}lder continuous in $z$ and
independent of $y$, then (\ref{equ0}) has a unique solution. In this
note, we will prove this conjecture under a more general
condition---uniform continuity---instead of H\"{o}lder continuity.
In other words, $g$ satisfies the following condition:
\begin{description}
\item[(H1).] $g(\oo,t,\cdot)$ is uniformly continuous and uniformly
with respect to $(\oo,t)$, i.e., there exists a function $\phi$ from
$\Real_+$ to itself, which is continuous, non-decreasing,
subadditive and of linear growth, and $\phi(0)=0$ such that,
$$\abs{g(\omega,t,z_1)-g(\omega,t,z_2)}\le\phi(\abs{z_1-z_2}),\ P-a.s.,\
\mbox{$\forall t\in[0,T],z_1,z_2\in\Real^d$}.$$ Here we denote the
constant of linear growth of $\phi$ by $A$, i.e., $$0\le\phi(x)\le
A(x+1)$$ for all $x\in\Real_+$ (see Crandall \cite{crandall1997}).
Moreover $(g(t,0))_{t\in[0,T]}$ is assumed to be bounded.
\end{description}
\begin{remark}
Clearly (H1) implies (H1'):
\begin{description}
\item [(H1').] $g(\oo,t,\cdot)$ is continuous, and of linear growth, i.e.,
there exists a positive real number $B$, such that
$$\abs{g(\oo,t,z)}\le B(\abs{z}+1),\quad P-a.s.,\ \mbox{for all
$(t,z)\in[0,T]\times\Real^d$}.$$
\end{description}
According to the result in \cite{LM1}, (H1') guarantees the
existence of a solution of (\ref{equ0}).
\end{remark}
This note is organized as follows. In Section 2 we formulate the
problem accurately and give some preliminary results. Finally,
Section 3 is devoted to the proof of the main theorem.

\section{Preliminaries}
Let $(\Omega,\mathcal{F},P)$ be a probability space and $W$ be a
$d$-dimensional standard Brownian motion on this space. Let
$(\mathcal{F}_{t})_{t\geq0}$ be the filtration generated by this
Brownian motion: $\mathcal{F}_{t}=\sigma \left \{W_{s},s\in \lbrack
0,t]\right \} \cup \mathcal{N}$,
$\mathbb{F}=(\mathcal{F}_{t})_{t\geq0}$, where $\mathcal{N}$ is the
set of all $P$-null subsets.

Let $T>0$ be a fixed real number. In this note, we always work in
the space $(\Omega,\mathcal{F}_{T},P)$. For a positive integer $n$
and $z\in \mathbb{R}^{n}$, we denote by $\left \vert z\right \vert $
the Euclidean norm of $z$. We will denote by $\mathcal{H}_{n}^{2}=\mathcal{H}_{n}^{2}(0,T;\mathbb{R}%
^{n})$, the space of all $\mathbb{F}$--progressively measurable $\mathbb{R}%
^{n}$--valued processes such that $\Ee\left[\int_{0}^{T}\left \vert
\psi
_{t}\right \vert ^{2}\,dt\right]<\infty$, and by $\mathcal{S}^{2}=\mathcal{S}%
^{2}(0,T;\mathbb{R})$ the elements in $\mathcal{H}_{1}^{2}$ with
continuous paths such that $\Ee\left[\sup_{t\in [0,T]}\left \vert
\psi_{t}\right \vert ^{2}\right]<\infty$.

Now, let $\xi\in L^2(\Omega,\F_T,P)$ be a terminal value,
$g:\Omega\times[0,T]\times\Real^d\to\Real$ be the generator, such
that the process $g(\omega,t,z)_{t\in[0,T]}\in\Hh_{1}^{2}$ for any
$z\in\Real^d$. A solution of a BSDE is a pair of processes
$(y_t,z_t)_{t\in[0,T]}\in\Ss^2\times\Hh_d^2$ satisfying BSDE
(\ref{equ0}).

We now introduce a useful lemma which plays an important role in
this note. First we define
$$\underline{f}_n(z)\triangleq\inf_{u\in\mathbb{Q}^d}\{f(u)+n|z-u|\}\quad\mbox{and}\quad\bar{f}_n(z)\triangleq\sup_{u\in\mathbb{Q}^d}\{f(u)-n|z-u|\},$$
where $f$ satisfies (H1) and $n\in\mathbb{N}$. Also we define
$C=\max\set{A,B}$. Then one has
\begin{lemma}\label{inf-convolutionthm}
Let $f$ satisfy (H1) and $\bar{f}_n,\underline{f}_n$ be defined as
above. Then for $n>C$,

i).  $-C(\abs{z}+1)\le\ul{f}_n(t,z)\le f(t,z)\le\bar{f}_n(t,z)\le
C(\abs{z}+1)$ P-a.s. for any $(t,z)\in[0,T]\times\Real^d$;

ii). $\ul{f}_\cdot(t,z)$ is non-decreasing and $\bar{f}_\cdot(t,z)$
is non-increasing for any $(t,z)\in[0,T]\times\Real^d$;

iii). $\abs{\bar{f}_n(t,z_1)-\bar{f}_n(t,z_2)}\le n\abs{z_1-z_2}$
and $\abs{\ul{f}_n(t,z_1)-\ul{f}_n(t,z_2)}\le n\abs{z_1-z_2}$ P-a.s.
for any $t\in[0,T]$, $z_1,z_2\in\Real^d$;

iv). If $z^n\to z$ as $n\to\infty$, then $\ul{f}_n(t,z^n)\to f(t,z)$
and $\bar{f}_n(t,z^n)\to f(t,z)$ P-a.s. as $n\to\infty$;

v). $0\le f(t,z)-\ul{f}_n(t,z)\le\phi\left(\frac{2C}{n-C}\right)$
and $0\le \bar{f}_n(t,z)-f(t,z)\le\phi\left(\frac{2C}{n-C}\right)$
P-a.s. for any $(t,z)\in[0,T]\times\Real^d$.
\end{lemma}
\begin{proof}
It is not hard to check i)---iv) (see \cite{LM1}).

We now prove v). It follows from (H1) that, for given
$(t,z)\in[0,T]\times\Real^d$, one has
\begin{equation}\label{eqnadd01}
f(t,u)\ge f(t,z)-\phi(\abs{z-u})\ge f(t,z)-A(\abs{z-u}+1)\ge
f(t,z)-C(\abs{z-u}+1),
\end{equation}
for any $u\in\Real^d$. Given $n>C$, we define
$$\Lambda_n\triangleq\set{u\in\mathbb{Q}^d:n\abs{z-u}\ge
C(\abs{z-u}+2)}.$$ Clearly, $\Lambda_n$ is not empty and
$\mathbb{Q}^d=\Lambda_n\cup\Lambda_n^c$ where
$$\Lambda_n^c=\set{u\in\mathbb{Q}^d:n\abs{z-u}<C(\abs{z-u}+2)}$$ is
the complementary set of $\Lambda_n$ (which is not empty too). For
any $u\in\Lambda_n$, it follows from (\ref{eqnadd01}) that,
$$f(u)+n\abs{z-u}\ge f(u)+C(\abs{z-u}+2)\ge f(z)+C.$$
Then by i) of this lemma, one has for any $u\in\Lambda_n$,
$$f(t,u)+n\abs{z-u}>f(t,z)+\frac{C}{2}>f(t,z)\ge \inf_{v\in\Lambda_n\cup\Lambda_n^c}\set{f(t,v)+n\abs{z-v}}.$$
Therefore
\begin{eqnarray*}
\underline{f}_n(t,z)&=&\inf_{u\in\Lambda_n\cup\Lambda_n^c}\set{f(t,u)+n|z-u|}=\inf_{u\in\Lambda_n^c}\set{f(t,u)+n|z-u|}\\
&=&\inf\{f(t,u)+n|z-u|:u\in\mathbb{Q}^d\mbox{ and }n\abs{z-u}<
C(\abs{z-u}+2)\}\\
&\ge&\inf\{f(t,u):u\in\mathbb{Q}^d\mbox{ and }n\abs{z-u}<
C(\abs{z-u}+2)\}\\
&\ge&\inf\{f(t,z)-\phi(\abs{z-u}):u\in\mathbb{Q}^d\mbox{ and
}\abs{z-u}\le \frac{2C}{n-C}\}\\
&=&f(t,z)-\phi\left(\frac{2C}{n-C}\right).
\end{eqnarray*}
Analogously we can prove the second part of vi). The proof is
complete.
\end{proof}
\begin{remark}
If $f$ satisfies (H1), then for any $(t,z)\in[0,T]\times\Real^d$ and
$n>C$,
$$0\le\bar{f}_n(t,z)-\ul{f}_n(t,z)\le
2\phi\left(\frac{2C}{n-C}\right),\quad P-a.s.$$
\end{remark}
\section{Main Theorem}
To begin with, we introduce two sequences of BSDE as follows:
\begin{equation}\label{yilowersolution}
\ul{y}_t^{n}=\xi+\int_t^T{\ul{g}_n(s,\ul{z}_s^n)}ds-\int_t^T{\ul{z}_s^n}dW_s
\end{equation}
and
\begin{equation}\label{yisuppersolution}
\bar{y}_t^n=\xi+\int_t^T\bar{g}_n(s,\bar{z}_s^n)\,ds-\int_t^T\bar
{z}_s^n\,dW_s
\end{equation}

Clearly, for any given $n>C$, both (\ref{yilowersolution})
 and (\ref{yisuppersolution}) have unique
adapted solutions, for which we denote them by
$(\ul{y}^n_t,\ul{z}^n_t)_{t\in[0,T]}$ and
$(\bar{y}^n_t,\bar{z}^n_t)_{t\in[0,T]}$ respectively. Moreover we
denote the maximal solution and the minimal one of (\ref{equ0})
respectively by $(\bar{y}_t,\bar{z}_t)_{t\in[0,T]}$ and
$(\ul{y}_t,\ul{z}_t)_{t\in[0,T]}$, and any given solution of
(\ref{equ0}) by $(y_t,z_t)_{t\in[0,T]}$. We now have the following
lemma.\vspace{0.1cm}
\begin{lemma}\label{lowersolutionproperties1}
Let $g$ satisfy (H1) and $\xi\in L^2(\Oo,\F_T,P)$. Then one has,

i). For $t\in[0,T]$ and
$n>C$,
$$\bar{y}_t^n\ge\bar{y}_t^{n+1}\ge\bar{y}_t\ge
y_t\ge\ul{y}_t\ge\ul{y}_t^{n+1}\ge\ul{y}_t^{n},\quad P-a.s.$$
Moreover,
$$\Ee\left[\abs{\bar{y}_t^n-\bar{y}_t}^2\right]+\Ee\left[\int_0^T\abs{\bar{z}_t^n-\bar{z}_t}^2\,dt\right]\to
0$$ and
$$\Ee\left[\abs{\ul{y}_t^n-\ul{y}_t}^2\right]+\Ee\left[\int_0^T\abs{\ul{z}_t^n-\ul{z}_t}^2\,dt\right]\to
0$$ as $n\to\infty$;

ii). In addition, there exists some positive constant $M_0$
depending only on $C$, $T$ and $\xi$, such that
$$E\left[\abs{\bar{y}_t^n}^2\right]\le M_0\quad
E\left[\int_0^T\abs{\bar{z}_t^n}^2\,dt\right]\le M_0$$ and
$$\Ee\left[\abs{\ul{y}_t^n}^2\right]\le M_0,\quad
E\left[\int_0^T\abs{\ul{z}_t^n}^2\,dt\right]\le M_0$$ for any $n>C$;

iii). For any $n>C$,
$$\Ee\left[\abs{\bar{y}_t^n-\ul{y}_t^n}\right]\le
2\phi\left(\frac{2C}{n-C}\right)T.$$
\end{lemma}

\begin{proof}
The proofs of i) and ii) can be found in \cite{LM1}. We now prove
iii). Here we always assume $n>C$. By (\ref{yilowersolution}) and
(\ref{yisuppersolution}),
\begin{equation}\label{mainthm01}
\bar{y}_t^n-\ul{y}_t^n=\int_t^T(\bar{g}_n(s,\bar{z}_s^n)-\ul{g}_n(s,\ul{z}_s^n))\,ds-\int_t^T(\bar{z}_s^n-\ul{z}_s^n)\,dW_s,\quad
t\in[0,T].
\end{equation}
Note that
\begin{align*}
\bar{g}_n(s,\bar{z}_s^n)-\ul{g}_n(s,\ul{z}_s^n)&=\ul{g}_n(s,\bar{z}_s^n)-\ul{g}_n(s,\ul{z}_s^n)+\bar{g}_n(s,\bar{z}_s^n)-\ul{g}_n(s,\bar{z}_s^n)\\
&=\ul{g}_n(s,\bar{z}_s^n)-\ul{g}_n(s,\ul{z}_s^n)+\hat{g}_t^n,
\end{align*}
where
$\hat{g}_t^n:=\bar{g}_n(s,\bar{z}_s^n)-\ul{g}_n(s,\bar{z}_s^n)$. It
follows from v) of Lemma \ref{inf-convolutionthm} that $$0\le
\hat{g}_t^n\le 2\phi\left(\frac{2C}{n-C}\right),\quad P-a.s.\
\forall t\in[0,T].$$

We set $\hat{y}_t^n\triangleq\bar{y}_t^n-\ul{y}_t^n$,
$\hat{z}_t^n\triangleq\bar{z}_t^n-\ul{z}_t^n$, and denote by
$\bar{z}_t^{n,i}$,$\ul{z}_t^{n,i}$ the components of $\bar{z}_t^{n}$
and $\ul{z}_t^{n}$ respectively. Define
$$z_t^{n,0}\triangleq\bar{z}_t^{n},\quad
z_t^{n,i}\triangleq(\ul{z}_t^{n,1},\cdots,\ul{z}_t^{n,i},\bar{z}_t^{n,i+1},\cdots,\bar{z}_t^{n,d})$$
and $$b_t^{n,i}\triangleq
\1_{\set{\bar{z}_t^{n,i}\ne\ul{z}_t^{n,i}}}\frac{\ul{g}_n(t,z_t^{n,i-1})-\ul{g}_n(t,z_t^{n,i})}{\bar{z}_t^{n,i}-\ul{z}_t^{n,i}}.
$$
for $1\le i\le d$ where $\1$ is the indicator function. The equation
(\ref{mainthm01}) can rewritten as
$$\hat{y}_t^n=\int_t^T(b_s^n\hat{z}_s^n+\hat{g}_s^n)\,ds-\int_t^T\hat{z}_s^n\,dW_s,$$
for $t\in[0,T]$ where $b_s^n:=(b_s^{n,1},\cdots,b_s^{n,d})$
($i=1,\cdots,d$).

We now set $$q_t^n:=\exp\left[\int_0^t
b_s^n\,dW_s-\frac{1}{2}\int_0^t\abs{b_s^n}^2\,ds\right].$$ Since
$\ul{g}_n$ satisfies a Lipschitz condition, $\abs{b_s^n}\le n$ for
any given $n$. Applying It\^{o} formula to $q_t^n\hat{y}_t^n$ on
$[t,T]$ and then taking conditional expectation yields
\begin{align*}
\hat{y}_t^n&=(q_t^n)^{-1}\Ee\left[\int_t^T
q_s^n\hat{g}_s^n\,ds|\F_t\right]\\
&=\Ee\left[\int_t^T\exp\left(\int_t^s
b_r^n\,dW_r-\frac{1}{2}\int_t^s\abs{b_r^n}^2\,dr
\right)\hat{g}_s^n\,ds|\F_t\right].
\end{align*}
It follows from the property of exponential martingale that, for
$s\ge t$, $$\Ee\left[\exp\left(\int_t^s
b_r^n\,dW_r-\frac{1}{2}\int_t^s\abs{b_r^n}^2\,dr\right)\right]=1.$$
Therefore,
\begin{eqnarray*}
\Ee\left[\hat{y}_t^n\right]&=&\Ee\left[\Ee\left[\int_t^T\exp\left(\int_t^s
b_r^n\,dW_r-\frac{1}{2}\int_t^s\abs{b_r^n}^2\,dr\right)\hat{g}_s^n\,ds
|\F_t\right]\right]\\
&=&\Ee\left[\int_t^T\exp\left(\int_t^s
b_r^n\,dW_r-\frac{1}{2}\int_t^s\abs{b_r^n}^2\,dr
\right)\hat{g}_s^n\,ds\right]\\
&\le&2\phi\left(\frac{2C}{n-C}\right)\Ee\left[\int_t^T\exp\left(\int_t^s
b_r^n\,dW_r-\frac{1}{2}\int_t^s\abs{b_r^n}^2\,dr\right)\,ds\right]\\
&\le& 2\phi\left(\frac{2C}{n-C}\right)T.
\end{eqnarray*}
The proof is complete.
\end{proof}

The following result is our main theorem.\vspace{0.1cm}
\begin{theorem}\label{uniquenessthm}
Let $g$ satisfy (H1) and $\xi\in L^2(\Oo,\F_T,P)$. Then the solution
of (\ref{equ0}) is unique.
\end{theorem}
\begin{proof}
From Lemma \ref{lowersolutionproperties1}-iii), it follows that
$\Ee\left[\abs{\bar{y}_t^n-\ul{y}_t^n}\right]\to 0$ as $n\to\infty$
for $t\in[0,T]$. Therefore
\begin{eqnarray*}
\Ee\left[\abs{\bar{y}_t-\ul{y}_t}\right]\le
\Ee\left[\abs{\bar{y}_t-\bar{y}_t^n}\right]+\Ee\left[\abs{\bar{y}_t^n-\ul{y}_t^n}\right]+\Ee\left[\abs{\ul{y}_t^n-\ul{y}_t}\right]\to
0,
\end{eqnarray*}
as $n\to\infty$ for $t\in[0,T]$. The proof is complete.
\end{proof}
\begin{remark}
In the case when $g$ depends on $y$ and is uniformly continuous
condition in $y$, the uniqueness of solution does not hold in
general. For example, let us consider the following equation:
$$y_t=\int_t^1 \sqrt{\abs{y_s}}ds-\int_t^1 z_s\,dW_s\quad\mbox{ for
$t\in[0,1]$}.$$ Clearly, $g(y)=\sqrt{\abs{y}}$ is uniformly
continuous. It is not hard to check that for each $c\in[0,1]$,
$$\left(y_t,z_t\right)_{t\in[0,1]}=\left(\left[\max(0,\frac{c-t}{2})\right]^2,0\right)_{t\in[0,1]}$$ is a solution
of the above BSDE.

Certainly, if $g$ is Lipschitz continuous with respect to $y$ or
satisfies some kind of monotonic condition just like used in
\cite{pardoux1996}, the result in Theorem \ref{mainthm01} also holds
true, this point is not difficult to be found in the proofs of
Theorem \ref{mainthm01} and Lemma \ref{lowersolutionproperties1}.
\end{remark}
\begin{remark}
It is worth noting that there is an important difference between the
BSDE satisfying standard condition and the BSDE discussed in this
note: although we still have the associated comparison theorem for
this kind of BSDEs, the associated strict comparison theorem (see
\cite[(ii) of Proposition 2.1]{coquetHMP2002}) (which says, if
$\xi_1\ge\xi_2$ P-a.s. and $P(\xi_1>\xi_2)>0$, then
$y_0^{\xi_1}>y_0^{\xi_2}$ where
$(y_t^{\xi_i},z_t^{\xi_i})_{t\in[0,T]}$ denotes the solution of
$(g,T,\xi_i), i=1,2$) does not hold in general.

For example, let us consider a BSDE as follows:
$$y_t^X=X+\int_t^T\frac{3}{2}\abs{z_s^X}^{2/3}-\int_t^Tz_s^XdW_s,$$
where $W$ is a one-dimensional Brownian motion,
$g=\frac{3}{2}\abs{z}^{2/3}$. It is not hard to check that  for each
constant $c\in\Real$,
$$(y_t,z_t)_{t\in[0,T]}=\left(c-\frac{1}{4}W_t^4,-W_t^3\right)_{t\in[0,T]}$$ is
the solution of $(g,T,c-\frac{1}{4}W_T^4)$, hence
$y_0^{c-\frac{1}{4}W_T^4}=y_0^{c}=c$. But $c\ge c-\frac{1}{4}W_T^4$
P-a.s. and $P(c>c-\frac{1}{4}W_T^4)>0$. Economically, this means
that there exist infinitely many opportunities of arbitrage.

More detailed discussions about this phenomenon and the
corresponding PDE problem will appear in another paper.
\end{remark}

\section*{Acknowledgements}
The author would like to thank Professor S. Peng for his helps and
comments.

\end{document}